\documentclass[11pt,a4paper,reqno]{amsart}
\usepackage{color}
\usepackage{amsmath}
\usepackage{amsthm}
\usepackage{amsfonts}
\usepackage{hyperref}
\usepackage{amssymb}
\usepackage{comment}
\usepackage{bm}
\usepackage[shortlabels]{enumitem} 

\newcommand\vp{{\vec{\mathfrak p}}}
\newcommand\vb{{\vec{b}}}

\newcommand\vs{{\vec{s}}}
\newcommand\vv{{\vec{v}}}
\newcommand\vu{{\vec{u}}}
\newcommand\vw{{\vec{w}}}
\newcommand\vx{{\vec{x}}}

\newcommand\R{{\mathbf{R}}}

\newcommand\lead{{\textup{lead}}}

\newcommand\Z{{\mathbf{Z}}}

\newcommand\Q{{\mathbf{Q}}}

\theoremstyle{plain}
  \newtheorem{theorem}{Theorem}

  \newtheorem{proposition}{Proposition}
  
  \newtheorem{lemma}{Lemma}
  
  \newtheorem{corollary}{Corollary}

\theoremstyle{remark}
  \newtheorem{remark}[subsection]{Remark}
  \newtheorem{remarks}[subsection]{Remarks}
  \newtheorem{example}{Example}

\theoremstyle{definition}

\begin{document}

\title{Intersective polynomials and Diophantine approximation, II}

\author{Th\'ai Ho\`ang L\^e}
\address{T. H. L\^e, Department of Mathematics,
The University of Texas at Austin,
1 University Station, C1200
Austin, TX 78712}
\email{leth@math.utexas.edu}

\author{Craig V. Spencer}
\address{C. V. Spencer, Department of Mathematics, Kansas State University,
138 Cardwell Hall, Manhattan, KS 66506}
\email{cvs@math.ksu.edu}

\thanks{The research of the second author is supported in part by NSA Young Investigator Grants \#H98230-10-1-0155 and \#H98230-12-1-0220.}

\begin{abstract}
By applying Schmidt's lattice method, we prove results on simultaneous Diophantine approximation modulo 1 for systems of polynomials in a single prime variable provided that certain local conditions are met.
\end{abstract}

\maketitle

\section{Introduction}
This paper is a continuation of our paper \cite{lespen}, and we refer the reader to the introduction of that paper for a more detailed history of the problem. One common theme in Diophantine approximation is that of \textit{small fractional parts of polynomials}. It has been known since Vinogradov that for every positive integer $k$, there exists an exponent $\theta_k>0$ such that
\[
\min_{1 \leq n \leq N} \| \alpha n^{k} \| \ll_{k} N^{-\theta_k}
\]
for any positive integer $N$ and real number $\alpha$, where $\| \cdot \|$ denotes the distance to the nearest integer. Improved bounds for $\theta_k$ have been proved by many authors, but it is an open conjecture that we can choose $\theta_k$ to be $1-\epsilon$ for any $\epsilon >0$.

Thanks to the work of Danicic, Cook, Schmidt, Baker and others, Vinogradov's result has received many generalizations, which on a qualitative level take the following form:
\begin{theorem} \label{th:s}
There is an exponent $\theta=\theta(k,l)>0$ such that whenever $f_1, \ldots, f_l \in \R[x]$ are polynomials of degree at most $k$ without constant terms, we have
\[
\min_{1 \leq n \leq N} \max_{1 \leq j \leq l} \| f_{j}(n) \| \ll N^{-\theta}
\]
for every positive integer $N$.
\end{theorem}

In \cite{lespen}, we demonstrated that more general conditions than those of Theorem \ref{th:s} are  possible. The requirement that $f_1(0) = \cdots = f_l(0) = 0$ was dropped; however, an additional assumption is required to avoid local obstructions, which is easily seen to be necessary.

\begin{theorem} \label{th:main2}
Let $l$ be a positive integer and $h_1, h_2, \ldots, h_{k}$ be polynomials satisfying the following property. 

\begin{itemize}
\item[ ]
 If $f_{i}=\sum_{j=1}^{k} c_{ij} h_j$ for $i=1, \ldots, l$ are any $l$ linear combinations of $h_1, \ldots, h_k$  with coefficients $ 
c_{ij} \in \Z$, and $q$ is any non-zero integer, there exists $n \in \Z$ such that $f_{i}(n) \equiv 0 \pmod{q}$  for every  $i=1, \ldots, l$. 
\end{itemize}

\noindent
Then there is an exponent $\theta>0$ depending only on $l$ and the polynomials $h_i$ such that the following holds. 
Let $A$ be an arbitrary $l \times k$ matrix with real entries. Write 
$A \begin{pmatrix} h_1(n) \\ \vdots \\ h_{k}(n) \end{pmatrix} = \begin{pmatrix} v_1(n) \\ \vdots \\ v_{l}(n) \end{pmatrix}$. Then
\begin{equation*} \label{eq:A}
\min_{1\leq n \leq N} \max_{1\leq i \leq l} \|v_{i}(n)\| \ll_{h_1, \ldots, h_{k}} N^{-\theta},
\end{equation*}
where the bound is uniform in $N$ and $A$.
\end{theorem}
In another direction, in \cite{lespen} we also considered the question of obtaining Diophantine inequalities for polynomials in a single prime variable.
In the case of a single polynomial, we obtained the following result, with a similar necessary local condition.

\begin{theorem} \label{th:main3}
Let $h$ be a polynomial with the property that for every $q \neq 0$, there is an $n \in \Z$ such that $h(n) \equiv 0 \pmod{q}$ and furthermore $q$ is coprime to $n$. Then there is an exponent $\theta>0$, depending only on the degree of $h$, such that
\begin{equation*} \label{eq:P}
\min_{\substack{1 \leq p \leq N \\ p \,\, \mathrm{prime}}} \| \alpha h(p) \| \ll N^{-\theta} 
\end{equation*}
for any positive integer $N$ and real number $\alpha$.
\end{theorem}

In this paper, we prove the following prime analogue of Theorem \ref{th:main2} which generalizes Theorem \ref{th:main3}.  

\begin{theorem} \label{th:mainresult}
Let $l$ be a positive integer and $h_1, h_2, \ldots, h_{k}$ be polynomials satisfying the following property.

\begin{itemize}
\item[ ] If $f_{i}=\sum_{j=1}^{k} c_{ij} h_j$ for $i=1, \ldots, l$ are any $l$ linear combinations of $h_1, \ldots, h_k$ with coefficients $c_{ij} \in \Z$, and $q$ is any non-zero integer,
there exists $n \in \Z$ such that $q$ is coprime to $n$ and $f_{i}(n) \equiv 0 \pmod{q}$ for every $i=1, \ldots, l$. 
\end{itemize}

\noindent
Then there is an exponent $\theta>0$ depending only on $l$ and the polynomials $h_i$ such that the following holds. 
Let $A$ be an arbitrary $l \times k$ matrix with real entries. Write 
$A \begin{pmatrix} h_1(n) \\ \vdots \\ h_{k}(n) \end{pmatrix} = \begin{pmatrix} v_1(n) \\ \vdots \\ v_{l}(n) \end{pmatrix}$. Then
\begin{equation} \label{eq:A2}
\min_{\substack{1 \leq p \leq N \\ p \,\, \mathrm{prime}}} \max_{1\leq i \leq l} \|v_{i}(p)\| \ll_{h_1, \ldots, h_{k}} N^{-\theta},
\end{equation}
where the bound is uniform in $N$ and $A$.
\end{theorem}
%
\begin{example}
 Theorem \ref{th:mainresult} implies that, for some $\theta>0$, we have
 \[
 \min_{\substack{1 \leq p \leq N \\ p \,\, \mathrm{prime}}} \max (\| \alpha (p-1)^2\|, \| \alpha (p-1)^3\|) \ll N^{-\theta} 
 \]
 uniformly in $\alpha, \beta$ and $N$. 
\end{example}
\begin{example}
It is known (see Section \ref{sec:intersective}) that the polynomials $(x^3-19)(x^2+x+1)$ and $x(x^3-19)(x^2+x+1)$ satisfy the condition of Theorem \ref{th:mainresult}.
 Therefore, for some $\theta>0$, we have
  \[
 \min_{\substack{1 \leq p \leq N \\ p \,\, \mathrm{prime}}} \max (\| \alpha (p^3-19)(p^2+p+1)\|, \| \alpha p(p^3-19)(p^2+p+1)\|) \ll N^{-\theta} 
 \]
 uniformly in $\alpha, \beta$ and $N$.
\end{example}

\begin{remarks} $ $

\vspace{-.3cm}
\begin{itemize}   
\item One can see from the proof that $\theta$ actually depends at most on $k, l$ and $\displaystyle \max_{1\leq i \leq k} \deg{h_i}$ but not on the coefficients of $h_i$.
\item The divisibility condition on the polynomials $h_1, h_2, \ldots, h_{k}$ is necessary. This can be seen by taking $A$ to be $\frac{1}{q}$ times an integral $l \times k$ matrix
and $N$ sufficiently large in terms of $q$.
\end{itemize}
\end{remarks}
As pointed out in \cite{lespen}, Theorem \ref{th:mainresult} could have been obtained if one had a simultaneous approximation result 
pertaining to exponential sums over polynomials evaluated at primes, similar in spirit to results of Weyl and Vinogradov. In this paper we obtain such a result 
(Theorem \ref{th:simul}), which can be of independent interest.

One interesting feature of Theorem \ref{th:mainresult} is that we have polynomial savings in (\ref{eq:A2}). Typically, when dealing with exponential sums 
involving the primes less than $N$, one usually expects a savings which is polynomial in $\log N$.
In our situation, however, we work exclusively with minor arc estimates, which 
give us a savings of a power of $N$. Also, we work with primes in arithmetic progressions, but Linnik's theorem enables us to work 
with arithmetic progressions whose moduli are as big as a small power of $N$.

The plan of the paper is as follows. In Section \ref{sec:intersective} we discuss polynomials satisfying the condition in Theorem \ref{th:mainresult}, 
which we coin \textit{intersective polynomials of the second kind}. In particular we show how Berend and Bilu's procedure \cite{bb} of determining intersective polynomials
can be modified slightly to check if a given polynomial in $\Z[x]$ is intersective of the second kind. In Section \ref{sec:lemmas} we recall some preparatory lemmas. The
simultaneous approximation result, our main tool in the proof of Theorem \ref{th:mainresult}, is proved in Sections \ref{sec:weyl} and \ref{sec:simultaneous}.
Finally, in Sections \ref{sec:1dim} and \ref{sec:general} we prove Theorem \ref{th:mainresult} by induction on $l$.

\noindent \textbf{Notation.} We will use Vinogradov's notation $\gg$ and $\ll$ in this paper. The implied constants may depend on $k,l$ and the degrees of $h_1, \ldots, h_k$.

\noindent \textbf{Acknowledgement.} The authors would like to thank Roger Baker, Andrew Granville, Ben Green, and Terence Tao for helpful conversations.

\section{On intersective polynomials of the second kind} \label{sec:intersective}
\subsection{Definitions}
We now recall and define some terms, not all of which are standard in the literature,  regarding the polynomials of interest in this paper. 
A polynomial $h \in \Z[x]$ is called an \textit{intersective polynomial} (of the first kind) if it has a root $\pmod{q}$ for any integer $q \neq 0$. 
If a polynomial has an integer root, then it is certainly intersective, but there are examples of intersective polynomials without rational roots, 
such as $(x^3-19)(x^2+x+1)$ and $(x^2-13)(x^2-17)(x^2-221)$ \cite{bb}. Polynomials $(h_1, \ldots, h_k)$  are called \textit{jointly intersective} (of the first kind) if they have a common root $\pmod{q}$ for any integer $q \neq 0$.\footnote{In the introduction of \cite{lespen}, the definition of jointly intersective polynomials is not correctly stated. We would like to take this opportunity to point out this mistake and rectify it.}
It is easy to see that $(h_1, \ldots, h_k)$ are jointly intersective if and only if
their g.c.d. is intersective. The condition in Theorem \ref{th:main2} is closely related with the notion of jointly intersective polynomials. It says that any $l$ $\Z$-linear combinations of $h_1, \ldots, h_k$ are jointly intersective.
In \cite{lespen}, we showed that if $l \geq 2$, then this condition implies that $h_1, \ldots, h_k$ are  actually jointly intersective themselves, 
but this is not necessarily true if $l=1$.

Following \cite{le}, a polynomial $h \in \Z[x]$ is called an \textit{intersective polynomial of the second kind} if it has a root $\pmod{q}$ that is coprime to $q$ for any integer $q \neq 0$. 
These polynomials are referred to as $\mathcal{P}$-intersective polynomials in \cite{rice} and \textit{intersective polynomials along the primes} in \cite{wierdl}. 
A polynomial having 1 as a root is clearly intersective of the second kind, but again $(x^3-19)(x^2+x+1)$ and $(x^2-13)(x^2-17)(x^2-221)$ serve as examples of 
intersective polynomials of the second kind without rational roots. Polynomials $(h_1, \ldots, h_k)$  are called \textit{jointly intersective of the second kind} 
if they have a common root $\pmod{q}$ that is coprime to $q$ for any integer $q \neq 0$. Again, $(h_1, \ldots, h_k)$ are jointly intersective of the second kind if and only if
their g.c.d. is intersective of the second kind. Also, the condition in Theorem \ref{th:main2} implies that $h_1, \ldots, h_k$ are jointly intersective of the second kind if $l > 1$, but this is not necessarily true if $l=1$.

\subsection{Descriptions}
Similarly to intersective polynomials of the first kind (see \cite[Section 2]{lespen}), we have the following descriptions of intersective polynomials of the second kind.
A polynomial $h \in \Z[x]$ is intersective of the second kind if and only if it has a root 
in $\Z_p^{\times}$ for any prime $p$, where $\Z_p$ is the ring of $p$-adic integers and $\Z_p^{\times}$ is the set of invertible elements of $\Z_p$. 
We can fix, for each $d \in \Z^{+}$, an integer $-d < r_d \leq 0$  such that $(r_d,d)=1$ and $h(n) \equiv 0 \pmod{d}$ whenever $n \equiv r_d \pmod{d}$. Moreover, 
$r_{dq} \equiv r_d \pmod{d}$ for any $d, q \in \Z^{+}$.

The polynomials $h_1, \ldots, h_k \in \Z[x]$ are jointly intersective of the second kind if and only if they are all multiples of an intersective polynomial of the second kind $h \in \Z[x]$.
Thus $h_1, \ldots, h_k \in \Z[x]$ are jointly intersective of the second kind if and only if they have a common root in $\Z_p^{\times}$ for every prime $p$. Also, corresponding to any system of jointly intersective polynomial $(h_1, \ldots, h_k)$, we can associate a sequence $(r_d)_{d \in \Z^{+}}$ such that 
\begin{eqnarray} 
& & -d < r_d \leq 0,  \,(r_d,d)=1, \, r_{dq} \equiv r_d \pmod{d} \textrm{ for any } d, q \in \Z^{+}, \nonumber \\
& & \textrm{ and } \, h_i(n) \equiv 0 \pmod{d} \textrm{ for all }i=1, \ldots, k. \label{eq:rd} \label{eq:d}
\end{eqnarray}

We have the following results for intersective polynomials of the second kind, which are similar to \cite[Propositions 1 and 2]{lespen} and whose proofs will be omitted.

\begin{proposition} \label{prop:intersective}
 $ $

\vspace{-.3cm}
\begin{enumerate}
 \item Let $h_1, \ldots, h_k \in \Z[x]$ be polynomials such that any two linear combinations of $h_1, \ldots, h_k$ with coefficients in $\Z$ are jointly intersective of the second kind. 
Then $h_1, \ldots, h_k$ are jointly intersective of the second kind themselves.

\item For any natural number $k$, there exist polynomials $h_1, \ldots, h_k \in \Z[x]$ that are not jointly intersective of the second kind, but any linear combination of $h_1, \ldots, h_k$ with coefficients in $\Z$ is intersective of the second kind.
\end{enumerate}
\end{proposition}

In \cite{bb}, Berend and Bilu obtained a procedure to check if a given polynomial in $\Z[x]$ is intersective. 
Their procedure can be modified very slightly to check if a given polynomial is intersective of the second kind, which we will now describe. Let us first recall their procedure.
Let $P(x)=\sum_{i=1}^n a_i x^{i} \in \Z[x]$. We assume that $\gcd(a_0, \ldots, a_n)=1$. 
Suppose $P(x)=h_1(x)\cdots h_{\nu}(x)$ factors as a product of irreducible polynomials $h_i$ in $\Z[x]$. Let  $\rho_i$ denote  the resultant of $h_i$ and $h_i'$. 
Set $\delta=\rho_1\cdots \rho_\nu$. 

Let $L$ be the splitting field of $P$ over $\Q$ and $G=\text{Gal}(L/\Q)$. For each $i$, let $\theta_i$ be a fixed root of $h_i$, $K_i=\Q(\theta_i)$, and $H_i=\text{Gal}(L/K_i)$ be a subgroup of $G$.
Let $U=\cup_{i=1}^{\nu} H_i \subset G$.

If $p$ is a prime not dividing $\delta$, then $p$ is unramified in $L$. 
Let $\mathfrak{p}$ be any prime ideal of $L$ lying over $p$. 
Recall that the \textit{Frobenius symbol} $\left( \dfrac{L/\Q}{\mathfrak{p}} \right)$ is an element of $G$ satisfying 
\[
 \left( \dfrac{L/\Q}{\mathfrak{p}} \right) (\alpha) \equiv \alpha^{p} \pmod{\mathfrak{p}}
\]
for any $\alpha$ in the ring of integers of $L$. Finally, the \textit{Artin symbol} $\left[ \dfrac{L/\Q}{p} \right]$ 
is the conjugacy class of $\left( \dfrac{L/\Q}{\mathfrak{p}} \right)$ in $G$, which depends on $p$ but not on $\mathfrak{p}$.

Berend and Bilu's main observations are the following.

\begin{proposition}\label{prop:bb}

 $ $
 
\vspace{-.3cm}  
\begin{enumerate}
 \item \label{prop:p1} If $p$ is a prime not dividing $\delta$, then $P(x)$ has a root in $\Z_p$ if and only if  $\left[ \dfrac{L/\Q}{p} \right] \cap U \neq \emptyset$.  
 \item \label{prop:p2} For any prime $p$, $P(x)$ has a root in $\Z_p$ if and only if there is an integer 
$\lambda$ such that 
\begin{equation}\label{eq:p-adic}
|P(\lambda)|_p < |\delta|_{p}^2. 
\end{equation}
 In fact, this $\lambda$ refines to a root of $P$ in $\Z_p$.
\end{enumerate}

\end{proposition}

By the Chebotarev density theorem, for any conjugacy class $C \subset G$, there is a subset of primes $p$ of positive density (namely, equal to $\frac{|C|}{|G|}$) for which  $\left[ \dfrac{L/\Q}{p} \right]=C$.  
Thus in order to check if $P$ is intersective, it suffices to check the following:
\begin{itemize}
\item (\ref{eq:p-adic}) holds for any prime $p$ dividing $\delta$. In other words, for any $p|\delta$, $P(x)$ has a root $\pmod{p^{2\alpha+1}}$, where $p^{\alpha}|| \delta$. 
\item For any conjugacy class $C \subset G, C \cap U \neq \emptyset$. 
\end{itemize}
By the effective version of the Chebotarev density theorem, any conjugacy class $C$ can be realized as $\left[ \dfrac{L/\Q}{p} \right]$ for some $p \leq 2d_{L}^{A}$, 
where $d_L$ is the absolute discriminant of the field $L$ and $A$ is an effectively computable constant. Therefore, the second condition above can again be checked 
using a finite number of congruences.

If $P$ is intersective of the second kind, then $P$ is \textit{a fortiori} intersective, and we can use the above procedure to check if $P$ is intersective. 
Our observation is the following
\begin{center}(3) \textit{
If $P$ has a root in $\Z_p$, then this root is automatically in $\Z_p^{\times}$ unless $p | a_0$.}
\end{center}
Thus all we have to do is to check that for any prime $p$ dividing $a_0$, there is an integer $\lambda$ coprime to $p$ such that (\ref{eq:p-adic}) holds.
In other words, if $p|a_0$ and $p^{\alpha}|| \delta$, we have to check if $P$ has a root $\pmod{p^{2\alpha+1}}$ that is coprime to $p$. 

\begin{example}
 In \cite[Example 2]{bb}, it is shown that the polynomial 
 \[
  P(x)=(x^3-19)(x^2+x+1)
 \]
is intersective. In this case, $a_0=19$ and $\delta=3^4 \cdot 19^2$. It is also shown that the congruence
\[
 x^2+x+1 \equiv 0 \pmod{19^3} 
\]
has a solution. Clearly, this solution is coprime to 19. Therefore, $P$ is intersective of the second kind.
\end{example}

\begin{example}
It is known \cite[p. 3]{bs} that
\[
 P(x)=(x^2-13)(x^2-17)(x^2-221)
\]
is intersective of the first kind. In this case, $\delta=2^6 \cdot 13^2 \cdot 17^2$ and $a_0=-13^2 \cdot 17^2$.

The congruences $x^2-13 \equiv 0 \pmod{17^5}$ and $x^2-17 \equiv 0 \pmod{13^5}$ have solutions which are coprime to 17 and 13, respectively. 
Therefore, $P$ is intersective of the second kind.
\end{example}

We remark that there are polynomials without integer roots which are intersective of the first kind but not of the second kind. For example, it is shown in \cite{rabayev}
that the polynomial
\[
P(x)=(x^4-5x^2+x+4)(x^3-10x^2+9x-1)
\]
is intersective. However, 1 is not a root of $P$ mod 2, and $P$ is not intersective polynomials of the second kind.

\section{Preparatory lemmas} \label{sec:lemmas}
Throughout this paper, we will be working with primes in congruence classes and the weights
\begin{equation} \label{eq:lambda}
\lambda_{m,b}(n)= \begin{cases}
                                                                     \log(mn+b), & \hbox{if  $mn+b$ is prime,} \\
                                                                     0, & \hbox{otherwise,}
                                                              
                                                                 \end{cases}
\end{equation}
where $\gcd(b,m)=1$ and $0 \leq b < m$.

We have the following lemma, which is useful for moduli as large as a small power of $N$.

\begin{lemma}\label{lem:linnik}
 We have
 \[
  Nm^{-2} \ll \sum_{n=1}^N \lambda_{m,b}(n) \ll Nm
 \]
for some constant $L>0$, whenever $m \leq N^{1/L}$.
\end{lemma}
\begin{proof}
 The upper bound follows trivially from the prime number theorem, and the lower bound follows from the proof of Linnik's theorem (see, for instance, \cite[Chapter 18]{iwaniec}). 
\end{proof}

Next we recall some lemmas we used in \cite{lespen} which we will use again in this paper. The first two lemmas relate Diophantine inequalities to exponential sums.

\begin{lemma}[{\cite[Theorem 2.2]{baker}}] \label{lem:montgomery}
Let $M$ and $N$ be a positive integers. Consider a sequence of real numbers $x_1, \ldots, x_N$ and weights $c_1, \ldots, c_{N} \geq 0$. 
Suppose that $\|x_{i}\| \geq M^{-1}$ for all $i=1, \ldots, N$. Then there exists $1 \leq t \leq M$ such that
\[
 \left| \sum_{n=1}^{N} c_{n} e(t x_{n})\right| \geq \frac{1}{6M} \displaystyle \sum_{n=1}^{N}c_n.
\]
\end{lemma}

The next lemma is a multidimensional generalization of the previous one. It is a consequence of \cite[Lemma 5]{harman}. 
The same result could be obtained from the machinery in \cite{bmv} with better constants.

\begin{lemma} \label{lem:alternative}
Suppose $l>1$, $\Lambda$ is a lattice of full rank in $\R^{l}$, and $B_{l}$ is the unit ball in $\R^{l}$. 
Consider $\vx_1, \ldots, \vx_N \in \R^{l}$ and weights $c_1, \ldots, c_{N} \geq 0$. Let
\[
 S_{\vp}=\sum_{n=1}^{N} c_n e(\vx_n \cdot \vp).
\]
Suppose $\vx_i \not \in \Lambda + B_l$ for all $i=1, \ldots, N$. Then there is a primitive point $\vp$ in the dual lattice $\Pi$ of $\Lambda$ such that $|\vp|\ll_{l} 1$ and an integer $1 \leq t \ll_{l} \frac{1}{|\vec{p}|}$ such that
\begin{equation} \label{eq:st}
 |S_{t \vp}| \gg_{l} \det(\Lambda)^{-1} \sum_{n=1}^{N} c_n   .
\end{equation}
\end{lemma}
\begin{proof}
 This follows immediately from \cite[Lemma 5]{harman}, together with the observation that the number of points  $\vp \in \Pi$ satisfying $|\vp| \ll_l 1$ is $\ll_{l} \det(\Lambda)^{-1}$ (\cite[Lemma 14B]{schmidt}).
\end{proof}

\begin{remark}
 In \cite[Lemma 5]{harman}, the result is stated for the hypercube $[-1,1]^{l}$ instead of $B_l$, but for the purpose of induction, it is more convenient to work with the unit ball 
 since the intersection of a ball and a hyperplane is still a ball.
\end{remark}

In \cite{lespen}, we used a weaker version of Lemma \ref{lem:alternative} proved by Schmidt, in which we have $|\vp|\leq N^{\epsilon}$ 
and the sum on the right hand side of (\ref{eq:st}) is off by a factor of $N^{\epsilon}$. 
The added strength of Lemma \ref{lem:alternative} is immaterial for our purpose, but we mention it since it can be useful in other circumstances 
(e.g., when one does not expect polynomial savings for  Diophantine inequalities).

The next lemma modifies a systems of polynomials so that the coefficients interact in a nice way. 
We say that a system of polynomials $(g_1, \ldots, g_k)$ is \textit{nice} if $\deg g_1 < \deg g_2 < \cdots < \deg g_{k}$ and the coefficient of $x^{\deg g_{i}}$ in $g_{j}$ is 0 for $i \neq j$ (it is trivially 0 if $j<i$). 

\begin{lemma}[{\cite[Lemma 1]{lespen}}] \label{lem:change}
Suppose $r, d \in \Z$ and  $f_1, \ldots, f_k \in \Z[x]$ with $\deg f_1 < \deg f_2 < \cdots < \deg f_{k}$. There exists a $k \times k$ matrix $T$ and polynomials $g_1, \ldots, g_k \in \Z[x]$, depending on $d$ and $r$, satisfying the following properties:
\begin{enumerate}
	\item \label{p1} $T \begin{pmatrix} f_1(dx+r) \\ \vdots \\ f_{k}(dx+r) \end{pmatrix} = \begin{pmatrix} g_1(x) \\ \vdots \\ g_{k}(x) \end{pmatrix}$
	\item \label{p2} $T$ is lower triangular with integer entries. All its diagonal entries are equal to $c$, where $c$ is an integer constant $($depending only on the the coefficients of $f_i)$.  
Actually, the $(i,j)$ entry of $T$ is $c_{ij}r^{\deg f_j-\deg f_i}$ if $i \leq j$ and $0$ otherwise, where $c_{ij}$ is an integer depending only on the coefficients of $f_1, \ldots, f_k$. 
	\item \label{p3} $g_1, \ldots, g_k$ form a nice system. Also, $\lead(g_i)=c d^{\deg f_i}\lead(f_i)$ for all $1 \leq i \leq k$.  
\end{enumerate}
\end{lemma}

Also, without loss of generality, we will assume throughout this paper that
polynomials $h_1, \ldots, h_k$ in Theorem \ref{th:mainresult} are of distinct degree. Indeed, let $g_1,\ldots, g_s \in \Z[x]$ be polynomials of distinct degree 
which form a basis of the $\Z$-module generated by $h_1, \ldots, h_k$. Then there is a $k \times s$ integer matrix $M$ such that  
\[ \begin{pmatrix} h_1(n) \\ \vdots \\ h_{k}(n) \end{pmatrix} = M \begin{pmatrix} g_1(n) \\ \vdots \\ g_{s}(n) \end{pmatrix}\]
It is easy to see that if $h_1, \ldots, h_k$ satisfy the condition of Theorem \ref{th:mainresult}, then so do $g_1,\ldots, g_s$. Also, if $g_1,\ldots, g_s$ satisfy
the conclusion of Theorem \ref{th:mainresult}, then so do $h_1, \ldots, h_k$.

\section{A Weyl-type estimate} \label{sec:weyl}
Throughout this section and the next, we write $f(x)=\alpha_1 x+\cdots+\alpha_k x^k \in \R[x]$, where $\alpha_k \neq 0$. We first recall a Weyl-type estimate for an exponential sum over prime numbers.

\begin{lemma}
\label{harman}
Suppose that $q \in \mathbb{N}$ and $\|q\alpha_k \|<q^{-1}$.  For $k>1$,
$$\sum_{\substack{  1 \leq p \leq N \\ \text{$p$ \rm is prime}}} (\log p) \, e(f(p)) \ll N^{1+\epsilon} \left(q^{-1} +N^{-1/2}+qN^{-k} \right)^{4^{1-k}},$$
and for $k=1$,
$$\sum_{\substack{  1 \leq p \leq N \\ \text{$p$ \rm is prime}}} (\log p) \, e(f(p)) \ll N (\log N)^4 \left( q^{-1/2}+N^{-1/5}+N^{-1/2}q^{1/2} \right).$$
\end{lemma}

\begin{proof}
This is a combination of \cite[Theorem 1]{harman2} due to Harman and \cite[Theorem 3.1]{vaughanbook} due to Vinogradov.
\end{proof}

We now prove the following Weyl-type estimate which is applicable to exponential sums over primes in arithmetic progression.  The proof is a generalization of the approach in \cite[Section 6]{ruzsa-sanders}.  In the linear case, one could instead use an exponential sum estimate by Balog and Perelli \cite{balog}.

\begin{theorem} \label{th:weyl}
Suppose that $m \leq N$ and that $a$ and $q$ are integers such that $(a,q)=1$ and $\left| \alpha_k - \frac{a}{q} \right| \leq \frac{1}{q^2}$.
For $k>1$,
$$\sum_{n=1}^N \lambda_{m,b}(n) e(f(n)) \ll (Nm)^{1+\epsilon} \left(q^{-1} +(Nm)^{-1/2}+qN^{-k} \right)^{4^{1-k}},$$
and for $k=1$,
$$\sum_{n=1}^N \lambda_{m,b}(n) e(f(n)) \ll Nm (\log N)^4 \left( q^{-1/2}+(Nm)^{-1/5}+N^{-1/2}q^{1/2} \right).$$
\end{theorem}

\begin{proof} 
Note that 
\begin{equation}
\label{detector-trick}
\sum_{n=1}^N \lambda_{m,b}(n) e(f(n)) = m^{-1} \sum_{l=0}^{m-1} \sum_{\substack{  1 \leq p \leq Nm+b \\ \text{$p$ \rm is prime}}} (\log p) \, e\left(f\left( \frac{p-b}{m} \right) + l \left(\frac{p-b}{m} \right) \right).
\end{equation}
Let $0 \leq l < m$ and   $$g(x) = \beta_0+\beta_1 x + \cdots + \beta_k x^k = f\left( \frac{x-b}{m} \right) + l \left(\frac{x-b}{m} \right).$$  
Were $\alpha_k \in \mathbb{Z}$, then $q=1$, and the result would follow trivially from the prime number theorem.
Unless $k=1$ and $\alpha_k +l=0$, $g(x)$ is a degree $k$ polynomial with leading coefficient 
$$\beta_k =\begin{cases}   \frac{\alpha_k}{m^k}, & \text{if $k \geq 1$,} \\\frac{\alpha_1+l}{m}, & \text{if $k=1$.}
\end{cases}$$
By Dirichlet's approximation theorem, there exist an integer $a'$ and a natural number $q'\leq Q := 2m^kq$ satisfying $(a',q')=1$ and $\left|\beta_k - \frac{a'}{q'} \right| < \frac{1}{Qq'} \leq \frac{1}{q'^2}$.

\textbf{Case 1:} Suppose that $k>1$.  We have that
\begin{align*}
\left|\frac{a'}{q'}-\frac{a}{m^kq} \right| & \leq \left|\frac{a'}{q'}-\frac{\alpha_k}{m^k} \right| + \left|\frac{\alpha_k}{m^k}-\frac{a}{m^kq} \right|
\leq \frac{1}{Qq'} + \frac{1}{q^2m^k}. 
\end{align*}
Upon multiplying each side of the above inequality by $qq'm^k$, we find that
$$|a'qm^k-aq'| \leq \frac{m^k q}{Q} + \frac{q'}{q} = \frac{1}{2}+\frac{q'}{q}.$$
Since $a'qm^k-aq'$ is an integer, if $q' < \frac{q}{2}$, then  
$a'qm^k=aq'$.  Because $(a,q)=1$, this would imply that $q \mid q'$ and $q \leq q'$.  Thus, we must have $\frac{q}{2} \leq q'$.  We may bound the inner sum of (\ref{detector-trick}) by applying Lemma \ref{harman} with the Diophantine approximation $\frac{a'}{q'}$ to the leading coefficient $\beta_k$ of $g(x)$, and we obtain
\begin{align*}
\sum_{n=1}^N \lambda_{m,b}(n) e(f(n)) &\ll 
(Nm+b)^{1+\epsilon} \left(q'^{-1} +(Nm+b)^{-1/2}+q'(Nm+b)^{-k} \right)^{4^{1-k}}
\\ & \ll (Nm)^{1+\epsilon} \left(q^{-1} +(Nm)^{-1/2}+qN^{-k} \right)^{4^{1-k}}. 
\end{align*}

\textbf{Case 2:} Suppose that $k=1$.  Similar to our work in Case 1, by bounding $\left|\frac{a'}{q'}-\frac{a+lq}{mq} \right|$, one can show that $\frac{q}{2} \leq q'$, and the result follows by applying Lemma \ref{harman} to the inner sum of (\ref{detector-trick}).
\end{proof}

\begin{corollary}\label{cor:weyl}
Suppose $1 \leq m, M \leq N^{c}$ and
 \[
  \left| \sum_{n=1}^N \lambda_{m,b}(n) e(f(n)) \right| \geq \frac{N}{M}.
 \]
Then there is an integer $1\leq q \ll (Mm)^{C} N^\epsilon$ such that
$$\| q \alpha_k \| \ll (Mm)^C N^{\epsilon-k}.$$
Here $C$ and $c$ are  constants depending only on $k$.
\end{corollary}

\begin{proof}
This follows from Theorem \ref{th:weyl} and  Dirichlet's approximation theorem.  One finds that $C=2$ for $k=1$ and $C=4^{k-1}$ for $k>1$.
\end{proof}

\begin{corollary}\label{cor:weyl2}
Suppose that $1\leq L \leq N$, $1 \leq m,$ $M \leq N^{c}$, and $\alpha$ is such that
\[
  \left| \sum_{n=L+1}^{N} \lambda_{m,b}(n) e(f(n)) \right| \geq \frac{N}{M}.
\]
Then there is an integer $1\leq q \ll (M m)^{C} N^\epsilon$ such that
$$\| q \alpha_k \| \ll  (M m)^C N^{\epsilon-k}.$$ Here $C$ and $c$ are  constants depending only on $k$.
\end{corollary}
\begin{proof}

Clearly,
\begin{equation*} \label{eq:sum}
  \left| \sum_{n=L+1}^{N} \lambda_{m,b}(n) e(\alpha n) \right| \leq  \left| \sum_{n=1}^{N} \lambda_{m,b}(n) e(\alpha n) \right| + 
 \left| \sum_{n=1}^{L} \lambda_{m,b}(n) e(\alpha n) \right|.
\end{equation*}
By Lemma \ref{lem:linnik}, there exists a positive constant $U$ such that the second term is bounded by  $ULm$.  We split our analysis into two cases. 

\noindent \textbf{Case 1}: Suppose that $L \leq \frac{N}{2MUm}$. Then $ULm \leq \frac{N}{2M}$, and
\[
\frac{N}{M} \ll \left| \sum_{n=1}^{N} \lambda_{m,b}(n) e(\alpha n) \right|. 
\]
By Corollary \ref{cor:weyl}, there is an integer $1\leq q \ll (M m)^{C} N^\epsilon$ such that $\| q \alpha_k \| \ll (Mm )^C N^{\epsilon-k}$.

\noindent \textbf{Case 2}: $N \geq L \geq \frac{N}{2MUm}$. We are in one of the following two subcases.
\begin{enumerate}[(a)]
\item
\[
\left| \sum_{n=1}^{N} \lambda_{m,b}(n) e(\alpha n) \right| \geq \frac{N}{2M}.
\]
As in Case 1, the result follows.

\item
\[
\left| \sum_{n=1}^{L} \lambda_{m,b}(n) e(\alpha n) \right| \geq \frac{N}{2M} \geq \frac{L}{2M}.
\]
Again, by Corollary \ref{cor:weyl}, there is an integer $1\leq q \ll (M m )^{C} L^\epsilon \leq (M m)^{C} N^\epsilon $ such that 
\[
\| q \alpha_k \| \ll (Mm)^C L^{\epsilon-k} \ll (Mm)^C \left( \frac{N}{Mm} \right) ^{\epsilon-k}  \ll (Mm)^{C+k-\epsilon} N^{\epsilon-k}.
\]
\end{enumerate}
In each case, we have the desired result.
\end{proof}

\section{A simultaneous approximation result} \label{sec:simultaneous}
The main result of this section is the following theorem, which is of independent interest.

\begin{theorem}\label{th:simul}
Suppose that $1 \leq m, M \leq N^{c}$ and
 \[
  \left| \sum_{n=1}^N \lambda_{m,b}(n) e(f(n)) \right| \geq \frac{N}{M}.
 \]
Then there is an integer $1\leq q \ll (Mm )^{C} N^\epsilon$ such that
$$\| q \alpha_j \| \ll (Mm)^C N^{\epsilon-j}$$
for every $j=1, \ldots, k$.  Here $C$ and $c$ are  constants depending only on $k$.
\end{theorem}

We will prove the following proposition by downward induction on $i$. When $i=1$, we recover Theorem \ref{th:simul}.

\begin{proposition} \label{prop:induct}
Let $i$ be an integer such that $1 \leq i \leq k$. Suppose that $M,m \leq N^c$ and 
\begin{equation} \label{eq:sum2}
  \left| \sum_{n=1}^N \lambda_{m,b}(n) e(f(n)) \right| \geq \frac{N}{M}.
\end{equation}
Then there is an integer $1\leq q \ll (Mm )^{C} N^\epsilon$ such that
$$\| q \alpha_j \| \ll  (Mm )^C N^{\epsilon-j}$$
for every $j=i, \ldots, k$.  Here $C$ and $c$ are  constants depending on $k$ and $i$.
\end{proposition}

\begin{proof}
 When $i=k$, this is Corollary \ref{cor:weyl}. Suppose $1 \leq i < k$
 and that the proposition is true for $i+1$. That is, there are 
 integers $1 \leq q \ll (Mm)^{C_1} N^\epsilon$ and $a_{i+1}, \ldots, a_k$ such that 
\[
\left| \alpha_j-\frac{a_j}{q} \right| \leq (Mm )^{C_1} N^{\epsilon-j} q^{-1}
\]
for every $j=i+1, \ldots, k$, where $C_1$ is a positive constant depending only on $k$ and $i$. Our goal is to extend this approximation to $\alpha_i$ as well.

Let $L=\lfloor N^{1-2\epsilon} (Mm )^{-C_2} \rfloor$ for some constant $C_2>C_1+1$. 
We partition $\{1,\ldots,N\}$ into congruence classes mod $q$, and we split each partition class into arithmetic progressions of length $L$. 
We have arithmetic progressions $P_1, \ldots, P_J$, all of length $L$ and common difference $q$. In particular $J=\frac{N}{L}+O(q)$. There are at most $qL$ 
numbers in $\{1, \ldots, N\}$ not belonging to any of these arithmetic progressions. The contribution of these numbers to the LHS of (\ref{eq:sum2}) is
\[
\ll qL \log N \ll  (Mm)^{C_1-C_2} N^{1-\epsilon}  \log N \ll \frac{N^{1-\epsilon}\log N}{M}.
\]
Therefore 
\[
 \sum_{j=1}^J \left| \sum_{n \in P_j} \lambda_{m,b}(n) e(f(n)) \right| \gg \frac{N}{M}.
\]
In particular there is a progression $P=\{t, t+q, \ldots, t+(L-1)q\} \subset \{1, \ldots, N\}$ such that
\[
 \left|\sum_{n \in P} \lambda_{m,b}(n) e(f(n)) \right| \gg \frac{L}{M}.
\]
 On the other hand, 
\begin{eqnarray*}
\left| \sum_{n \in P} \lambda_{m,b}(n) e(f(n)) \right| &=& \left| \sum_{l=0}^{L-1} \lambda_{m,b}(t+lq) e \left( \sum_{j=1}^{k} (t+lq)^j \alpha_j \right) \right|\\
  &=& \left| \sum_{l=0}^{L-1} \lambda_{m,b}(t+lq) e \left( g(l) + \sum_{j=i+1}^{k} ((t+lq)^j-t^j) \alpha_j \right) \right|\\
  &=& \left| \sum_{l=0}^{L-1} \lambda_{m,b}(t+lq) e \left( g(l) + \sum_{j=i+1}^{k} ((t+lq)^j-t^j) (\alpha_j-a_j/q) \right) \right|
\end{eqnarray*}
where we set
\[
g(l)=\sum_{j=1}^{i} \alpha_j (t+lq)^{j}.
\]
Note that 
\begin{eqnarray*}
\left|\sum_{j=i+1}^{k} ((t+lq)^j-t^j) (\alpha_j-a_j/q) \right| &\ll& \sum_{j=i+1}^k Lq N^{j-1} (Mm)^{C_1} N^{\epsilon-j}q^{-1}
\ll
 L (Mm )^{C_1} N^{\epsilon-1} .
\end{eqnarray*} 
Therefore, 
\begin{eqnarray*}
\left| \sum_{n \in P} \lambda_{m,b}(n) e(f(n)) \right| &\ll& \left| \sum_{l=0}^{L-1} \lambda_{m,b}(t+lq) e(g(l)) \right| + L (Mm )^{C_1} N^{\epsilon-1} \sum_{l=0}^{L-1} \lambda_{m,b}(t+lq)\\
 &\ll& \left| \sum_{l=0}^{L-1} \lambda_{m,b}(t+lq) e(g(l)) \right| + L^2 (Mm )^{C_1} N^{\epsilon-1} \log N . 
\end{eqnarray*}
Clearly, $$L^2 (Mm )^{C_1} N^{\epsilon-1} \log N \ll L (Mm)^{C_1-C_2} N^{-\epsilon} \log N < \frac{L N^{-\epsilon}\log N}{M }.$$ Therefore,
\begin{equation}\label{eq:alpha1}
\left| \sum_{l=0}^{L-1} \lambda_{m,b}(t+lq) e(g(l)) \right| \gg \frac{L}{M} \gg \frac{N}{M'} ,
\end{equation}
where $M'=(Mm)^{C_2+1}N^{2\epsilon}$.

Let us analyze the LHS of (\ref{eq:alpha1}), which is
\[
 \left| \sum_{l=0}^{L-1} \lambda_{m,b}(t+lq) e \left( \sum_{j=1}^{i} \alpha_j (t+lq)^{j} \right) \right|.
\]
Let $r$ be the remainder of $mt+b$ upon dividing by $mq$. Then we can write, for each $l$,
\[
m(t+lq)+b=mqn+r
\]
for some integer $n$ given by 
\[
t+lq=qn+\frac{r-b}{m}.
\]
Furthermore, as $l$ ranges from $0$ to $L-1$, $n$ ranges on an interval $[N'-L+1, N']$, 
for some $L-1 \leq N' \leq N$. Thus we can rewrite the LHS of (\ref{eq:alpha1}) as
\[
\left| \sum_{n=N'-L+1}^{N'} \lambda_{mq,r}(n) e(h(n)) \right|,
\]
where 
\[
h(n)= \sum_{j=1}^{i} \alpha_j  \left( qn+\frac{r-b}{m} \right)^{j}.
\]
Note that $h(n)$ is a polynomial of degree $i$ with leading coefficient $\alpha_i q^i$. We now invoke Corollary \ref{cor:weyl2}. We can find a constant $C_3>0$ and a $1 \leq q' \ll (M' mq )^{C_3}N'^\epsilon \leq (M'mq)^{C_3} N^\epsilon$ such that $\|q'q^i \alpha_i \| \ll  (M' m q)^{C_3} N'^{\epsilon-i} \ll (M' m q)^{C_3} L^{\epsilon-i} $.
By setting $q_0=q'q^i$ and recalling the definitions of $L$ and $M'$, we see that for some sufficiently large constant $C_4>0$, we have $1 \leq q_0 \ll ( Mm N^\epsilon)^{C_4}$ and $\|q_0 \alpha_j \| \ll ( Mm N^\epsilon)^{C_4} N^{-j}$
for $j=i, \ldots, k$.
\end{proof}

\section{The one-dimensional case} \label{sec:1dim}

In this section, we prove the one-dimensional case of Theorem \ref{th:mainresult}.

\begin{theorem}\label{th:1dim}
Suppose the polynomials $h_1, \ldots, h_k$ of distinct degree are such that any linear combination of them with integer coefficients is intersective of the second kind. 
Then there is an exponent $\theta>0$ (depending at most on $h_1, \ldots, h_k$) such that the following holds. If  $\alpha_1, \ldots, \alpha_{k}$ are arbitrary real numbers, then there exists 
a prime $1\leq p \leq N$ such that
\[
\| \alpha_1 h_1(p)+\cdots+\alpha_{k} h_{k}(p) \| \ll N^{-\theta},
\]
where the implied constant does not depend on $\alpha_1, \ldots, \alpha_k, N$.
\end{theorem}

\begin{proof}
By Theorem \ref{th:main3}, we may assume that $k \geq 2$.
Suppose for a contradiction that
for every prime $1 \leq p \leq N$, we have 
\[
\| \alpha_1 h_1(p)+\cdots+\alpha_{k} h_{k}(p) \| \geq M^{-1} ,
\]
where $M= \lfloor N^{\theta} \rfloor$ and $\theta$ is a sufficiently small exponent to be chosen later.
Let the matrix $T$ and polynomials $g_1, \ldots, g_k$ be obtained by applying Lemma \ref{lem:change} for the polynomials $h_1, \ldots, h_k$ with $d=1$ and $r=0$. 
Set
\[
 (\beta_1 \cdots \beta_k) = (\alpha_1 \cdots \alpha_k) T^{-1}.
\]
Then for all prime $p$, $1 \leq p \leq N$, we have
\begin{equation*} \label{eq:1}
\| \beta_1 g_1(p)+\cdots+\beta_{k} g_{k}(p) \| = \| \alpha_1 h_1(p)+\cdots+\alpha_{k} h_{k}(p) \| \geq M^{-1}.
\end{equation*}
By the Prime Number Theorem and Lemma \ref{lem:montgomery}, there exists $1 \leq t \leq M$ with
$$ \left| \sum_{\substack{p=1 \\ p \text{ prime}}}^N \log p \cdot e\left( t \beta_1 g_1(p)+\cdots+t\beta_{k} g_{k}(p) \right) \right| \gg \frac{N}{M}.$$ 
Recall that $\deg(g_i) = \deg(h_i) = d_i$ ($1 \leq i \leq k$).
Applying Theorem \ref{th:simul} to the polynomial
\[
 t \beta_1 g_1(n)+\cdots+t\beta_{k} g_{k}(n) 
\]
with $m=1$ and $b=0$, we see that (assuming $M<N^c$, where $c$ is the constant in Theorem \ref{th:simul}) there is an integer $1\leq q \leq M^{C}N^\epsilon$ such that
$$\| qt \beta_j \lead(g_j) \| \leq M^C N^{\epsilon-d_j}$$
for every $j=1, \ldots, N$, where $C$ is a constant depending only on $k$.

Let $R=qt \prod_{i=1}^{k} \lead (g_i)$, then $R \ll M^{C+1} N^\epsilon$. We have $\|R \beta_i\| \ll M^C N^{\epsilon-d_i}$ for all $i$. 
It follows that we can find integers $a_i$ such that $|\beta_i -\frac{a_i}{R}| \ll M^C N^{\epsilon-d_i} R^{-1} $ for all $i$. Let us now choose $1 \leq n \leq R$ such that $(n,R)=1$ and
\[
a_1g_1(n) + \cdots + a_k g_k(n) \equiv 0 \pmod{R} ,
\]
which is possible because each $g_i$ is a linear combination of the $h_i$ and consequently $a_1g_1 + \cdots + a_k g_k$ is intersective of the second kind. 

By Linnik's theorem, there is a prime $p$ such that $p \equiv n \pmod{R}$ and $p \ll R^{L}$, where $L$ is Linnik's constant. On the one hand, we have
\[
 p \ll R^{L} \ll \left( M^{C+1} N^\epsilon \right)^L \leq N
\]
if $\theta$ is sufficiently small and $N$ is sufficiently large. On the other hand,
\begin{eqnarray*}
 \left\| \sum_{i=1}^{k} \beta_i g_i(p) \right\| &\leq& \left| \sum_{i=1}^{k} \beta_i g_i(p) - \frac{1}{R} \sum_{i=1}^{k} a_i g_i(p) \right| \leq \sum_{i=1}^{k} \left| \left(\beta_i - \frac{a_i}{R} \right)g_i(p)\right|  \\
 & \ll & \sum_{i=1}^{k} M^C N^{\epsilon-d_i} R^{-1} R^{L d_i}  \ll M^C N^{\epsilon-1/L}.
\end{eqnarray*}
If $\theta$ is sufficiently small and $N$ is sufficiently large, then this is smaller than $M^{-1}$, which is a contradiction.
\end{proof}

For the remainder of the paper, we will study jointly intersective polynomials $h_1, \ldots, h_k$, and we will prove Theorem \ref{th:mainresult} by induction under this hypothesis. 
For technical reasons, we work with these polynomials along arithmetic progressions $dx+r_d$, where the sequence $(r_d)$ is given by (\ref{eq:d}), 
and the modulus $d$ may be as large as a small power of $N$. The base case is the following, whose proof we will omit since it is very similar to the proof of 
Theorem \ref{th:1dim} (see also \cite[Theorem 7]{lespen}), the difference being that we will work with the weights $\lambda_{d, r_d}$ in (\ref{eq:lambda}) 
instead of $\log p$. Also, one needs the lower bound in Lemma \ref{lem:linnik} when one applies Lemma \ref{lem:montgomery}.

\begin{theorem}\label{th:base} Let $h_1, \ldots, h_k$ be jointly intersective polynomials of the second kind of distinct degree.  
There are exponents $\theta, \sigma>0$ $($depending on the $h_i)$ such that the following holds. If $d$ is a modulus smaller than $N^{\sigma}$ and $\alpha_1, \ldots, \alpha_{k}$ are arbitrary real numbers, then there exists 
a prime $p$, $1\leq p \leq N$ such that $p \equiv r_d \pmod{d}$ and
\[
\| \alpha_1 h_1(n)+\cdots+\alpha_{k} h_{k}(n) \| \ll N^{-\theta},
\]
where the implied constant does not depend on $\alpha_1, \ldots, \alpha_k, N, d$.
\end{theorem}

\section{The general case} \label{sec:general}
In this section, we prove Theorem \ref{th:mainresult}. Following Schmidt \cite{schmidt}, we reformulate the problem in the language of lattices and prove a more general statement that 
allows us to perform induction. Precisely, we will prove the following theorem.

\begin{theorem}\label{lattice} Let $h_1, \ldots, h_k$ be polynomials satisfying the conditions in Theorem \ref{th:mainresult}.  For every natural number $l$, there are exponents $\theta_l, \sigma_l$ such that the following holds. If $\Lambda$ is a lattice with determinant $\det(\Lambda) \leq N^{\theta_l}$, $d$ is a modulus with $d \leq N^{\sigma_l}$, and $A$ is any real $l \times k$ matrix, 
then there exists a prime $p$, $1 \leq p \leq N$ such that 
\[ 
 A \begin{pmatrix} h_1(p) \\ \vdots \\ h_{k}(p) \end{pmatrix} \in \Lambda + B_l,
\]
where $B_l$ is the unit ball in $\R^{l}$. Furthermore, $p \equiv r_d \pmod{d}$.
\end{theorem}

It is easy to see that Theorem \ref{lattice} implies Theorem \ref{th:mainresult} (where $\theta$ can be taken to be $\theta_l/l$) by setting $\Lambda=N^{\theta_l/l} \Z^{l}$
and replacing the matrix $A$ in Theorem \ref{th:mainresult} with $N^{\theta_l/l} A$.

\begin{proof}[Proof of Theorem \ref{lattice}]
We prove this theorem by induction on $l$.  The case when $l=1$ follows from Theorem \ref{th:base}. Suppose that $l \geq 2$ and that we have determined appropriate values for $\theta_{l-1}$ and $\sigma_{l-1}$. Let $\theta_{l}$ and $\sigma_{l}$ be two exponents, which will be selected later, and suppose that
\begin{equation*} 
 A \begin{pmatrix} h_1(dn+r_d) \\ \vdots \\ h_{k}(dn+r_d) \end{pmatrix} \notin \Lambda + B_l
\end{equation*}
for all primes of the form $dn+r_d$ with $1 \leq n \leq M=\lfloor N^{1-\sigma_l} \rfloor -1$. Let $T$ be the matrix and $ g_1, \ldots, g_k$ be the  polynomials resulting from an application of Lemma \ref{lem:change} for the polynomials $h_1,\ldots, h_k$ with respect to $d$ and $r=r_d$.  Set $B=AT^{-1}$.  Then
\begin{equation} \label{b}
B \begin{pmatrix} g_1(n) \\ \vdots \\ g_{k}(n) \end{pmatrix} \notin \Lambda + B_l
\end{equation}
for all $n=1, \ldots, M$ with $dn+r_d$ prime. Let $\epsilon>0$ be any number, and assume that $N$ is sufficiently large in terms of $\epsilon$.  
Let $\vb_1, \ldots, \vb_k$ be the columns of $B$. Applying Lemma \ref{lem:alternative} to the vectors $\vx_n=g_1(n)\vb_1+\cdots+g_k(n)\vb_k$, we find a primitive point $\vp$
in the dual lattice $\Pi$ of $\Lambda$ with $|\vp| \ll 1$ and an integer $1 \leq t \ll 1$ such that 
\[
\left| \sum_{n=1}^{M} \lambda_{d,r_d}(n) e \left(t g_1(n) \vb_1 \cdot \vp + \cdots + t g_k(n) \vb_k \cdot \vp \right) \right| \gg \det(\Lambda)^{-1} \sum_{n=1}^{M} \lambda_{d,r_d}(n) \gg \det(\Lambda)^{-1} M d^{-2} . 
\]
Here, the last bound follows from Lemma \ref{lem:linnik}.
Applying Theorem \ref{th:simul}, 
there exists a constant $C>0$ such that we can find
$1\leq q \ll \left(\det(\Lambda) d^3 \right)^C M^\epsilon$ with
\[
\|q t \lead(g_i) \vb_i \cdot \vp  \| \ll \left(\det(\Lambda) d^3 \right)^C M^{\epsilon-d_i}
\]
for all $i=1, \ldots, k$, since $(g_1, \ldots, g_k)$ is a nice system with $d_i$ being the degree of $g_i$.
For any $\delta > 0$, by taking $\theta_l$ and $\sigma_l$ small enough depending on $\delta$, one can find $1\leq q \ll M^\delta$ with
$
\|q t \lead(g_i) \vb_i \cdot \vp  \| \ll M^{\delta-d_i}$
for all $i=1, \ldots, k$.

There are integers $n_i$ such that $\|q t \lead(g_i)\vec{b}_i \cdot \vp \|=|q t \lead(g_i) \vec{b}_i \cdot \vp-n_i|$. 
Since $\vp$ is a primitive point in $\Pi$, there are vectors $\vv_i \in \Lambda$ such that $n_i=\vp \cdot \vv_i$. 
Thus we have $|(q t \lead (g_i) \vb_i - \vv_i)\cdot \vp| \ll M^{\delta-d_i}$. 
Roughly speaking, this means that the vectors $\vu_i=q t \lead (g_i) \vb_i - \vv_i$ almost lie in the orthogonal complement of $\vp$, 
and we will exploit this fact to move to a space of smaller dimension.  
Before proceeding, let us state our goal. Let $R=c q t d^{d_k}\prod_{i=1}^{k}\lead(h_i)$, where $c$ is the constant in Property (\ref{p3}) of Lemma \ref{lem:change}.
Then $R$ is divisible by $q t \lead (g_i)$ for every $i$.

\textbf{Claim:} There exists $1\leq m \leq \sqrt{N} \leq M$ such each $g_i(m)$ is divisible by $R$, $dm+r_d$ is prime, and
\begin{equation} \label{goal}
\frac{g_1(m)}{q t \lead (g_1)} \vu_1+\cdots+\frac{g_k(m)}{q t \lead(g_k)} \vu_k \in \Lambda + B_l.
\end{equation}
Since  
\[
\frac{g_1(m)}{q t \lead (g_1)} \vu_1+\cdots+\frac{g_k(m)}{q t \lead (g_k)} \vu_k = g_1(m) \vb_1+\cdots+g_k(m) \vb_k-\frac{g_1(m)}{q t \lead (g_1)}\vv_1-\cdots-\frac{g_k(m)}{q t \lead (g_k)}\vv_k
\] 
and the $\vv_i$ are in $\Lambda$, this immediately implies that $g_1(m) \vb_1+\cdots+g_k(m) \vb_k \in \Lambda + B_l$, which contradicts (\ref{b}). (The claim makes it clear why we need to include the extra divisibility requirement in 
our induction hypothesis.)

Let $\Lambda'$ be the intersection of $\Lambda$ and the $(l-1)$-dimensional space $V=\vp^{\perp}$. Since $\vp$ is a primitive point in the dual lattice of $\Pi$, $\Lambda'$ is a sublattice of $\Lambda$ of dimension $l-1$. In order to achieve (\ref{goal}), we want to have
\begin{equation} \label{goal1}
\frac{g_1(m)}{qt \lead (g_1)}\vw_1+\cdots+\frac{g_k(m)}{q t \lead (g_k)}\vw_k \in \Lambda' + \frac{1}{2}B_l ,
\end{equation}
where $\vw_i$ is the orthogonal projection of $\vu_i$ onto $V$, and
\begin{equation} \label{goal2}
\frac{g_1(m)}{q t \lead (g_1)}(\vu_1-\vw_1)+\cdots+\frac{g_k(m)}{q t \lead (g_k)}(\vu_k-\vw_k) \in \frac{1}{2}B_l.
\end{equation}

Since $|\vp| \gg \det(\Lambda)^{-1}$ (see \cite[p. 28]{schmidt}), 
we have $$|\vu_i-\vw_i| = \frac{|\vu_i \cdot \vp|}{|\vp|} \ll \frac{M^{\delta-d_i}}{|\vp|} \ll \det(\Lambda) M^{\delta-d_i} \leq N^{\theta_l} M^{\delta-d_i}.$$
It is easy to see from Lemma \ref{lem:change} that we also have the bound $g_i(m)\ll (dm)^{d_i}$. 
Hence for every $1 \leq m \leq \sqrt{N}$,
\begin{eqnarray*}
\left| \frac{g_1(m)}{q t \lead (g_1)}(\vu_1-\vw_1)+\cdots+\frac{g_k(m)}{q t \lead (g_k)}(\vu_k-\vw_k) \right| &\leq& \sum_{i=1}^{k} N^{\theta_l} M^{\delta-d_i} (dm)^{d_i} \\
&\ll&  \sum_{i=1}^{k} N^{\theta_l} M^{\delta-d_i}  N^{d_i(1/2+\sigma_l)},
\end{eqnarray*}
which can be made smaller than 1/2 if $\theta_l$, $\sigma_l$, and $\delta$ are sufficiently small.
Thus (\ref{goal2}) holds for every $1 \leq m \leq N^{1/2}$.

Let us now turn our attention to (\ref{goal1}) and the conditions that each $g_i(m)$ is divisible by $R$  and that $dm+r_d$ is prime. These are satisfied if we can find a prime $1 \leq p \leq \sqrt{N}$ such that $p \equiv r_{R} \pmod{R}$ and
\begin{equation} \label{goal3}
h_1(p)\vs_1+\cdots+h_k(p)\vs_k \in \Lambda' + \frac{1}{2}B_l ,
\end{equation}
where 
\[
(\vs_1 \vs_2 \cdots \vs_k)= \left( \frac{\vw_1}{q t \lead (g_1)}  \quad  \frac{\vw_2}{qt \lead (g_2)}  \quad \cdots  \quad \frac{\vw_k}{q t \lead (g_k)} \right) T.
\]
Once such a prime $p$ is found, upon setting $m=\frac{p-r_d}{d}$, (\ref{goal1}) holds. 
Since $p \equiv r_R \, (\text{mod } R)$ and each $g_i(m)$ can be written as a linear combination of the $h_j(p)$, it follows that each $g_i(m)$ is divisible by $R$. Also,  note that since $R$ is divisible by $d$, we have $p \equiv r_d \, (\text{mod } d)$.

 By our induction hypothesis, we can find a prime $1 \leq p \leq N^{1/2}$ satisfying $p \equiv r_R \, (\text{mod } R)$ and (\ref{goal3})
 provided that $R \leq N^{\sigma_{l-1}/2}$ and $\det(\Lambda')\ll N^{\theta_{l-1}/2}$. 
By definition,
\[
R \ll qtd^{d_k} \ll  M^{\delta}  N^{d_k \sigma_l}, 
\]
and $\det(\Lambda')=|\vp|\det(\Lambda) \ll N^{\theta_l}$.
Hence, by choosing  $\delta$, $\theta_l$, and $\sigma_l$ sufficiently small in terms of $\sigma_{l-1}$ and $\theta_{l-1}$, our claim is achieved, and the theorem holds.
\end{proof}

\end{document}